\documentclass[10pt]{amsart}
\usepackage{amsmath,amssymb,amsfonts,amsthm}
\usepackage{mathtools}
\usepackage{verbatim}
\usepackage{graphicx}
\usepackage{color}
\usepackage{enumerate}
\usepackage{tikz}
\usetikzlibrary{calc}
\usepackage[T1]{fontenc}
\usepackage{hyperref,cleveref}
\usepackage[labelfont=up]{subcaption}
\usepackage[margin=3cm]{geometry}
\usepackage[normalem]{ulem}

\numberwithin{equation}{section}
\numberwithin{figure}{section}

\theoremstyle{plain}
\newtheorem*{reftheorem*}{Theorem}
\newtheorem*{refcor*}{Corollary}
\newtheorem{thm}{Theorem}[section]
\newtheorem{lemma}[thm]{Lemma}
\newtheorem{prop}[thm]{Proposition}
\newtheorem{corollary}[thm]{Corollary}
\theoremstyle{definition}
\newtheorem{defn}[thm]{Definition}
\newtheorem{example}[thm]{Example}
\newtheorem{remark}[thm]{Remark}

\newtheorem{notation}[thm]{Notation}

\newcommand{\R}{\mathbb{R}}
\newcommand{\N}{\mathbb{N}}

\renewcommand{\H}{\mathbb{H}}
\newcommand{\Hy}{\mathbb{H}}

\newcommand{\K}{\mathbb{K}}
\newcommand{\C}{\mathbb{C}}
\newcommand{\bP}{\mathbb{P}}

\newcommand{\cH}{\mathcal{H}}

\newcommand{\cN}{\mathcal{N}}

\newcommand{\cS}{\mathcal{S}}

\newcommand{\p}{\partial}

\newcommand{\SO}{\operatorname{SO}}
\newcommand{\SU}{\operatorname{SU}}
\newcommand{\Sp}{\operatorname{Sp}}
\newcommand{\Spin}{\operatorname{Spin}}

\definecolor{amethyst}{rgb}{0.6, 0.4, 0.8}

\title[Hyperbolic spaces with geometric $\&$ geometrically finite quasi-actions are symmetric]{Hyperbolic spaces with geometric and geometrically finite quasi-actions are symmetric}

\author[D.Groves]{Daniel Groves}
\address{Department of Mathematics, Statistics, and Computer Science,
University of Illinois at Chicago,
322 Science and Engineering Offices (M/C 249),
851 S. Morgan St.,
Chicago, IL 60607-7045}
\email{dgroves@uic.edu}

\author[E. Stark]{Emily Stark}
\address{Department of Mathematics, Wesleyan University, 265 Church St., Middletown, CT 06459}
\email{estark@wesleyan.edu}

\author[G.S. Walsh]{Genevieve S. Walsh}
\address{Department of Mathematics, Tufts University, Joyce Cummings Center, Medford, MA 02155}
\email{genevieve.walsh@tufts.edu}

\author[K. Whyte]{Kevin Whyte}
\address{Department of Mathematics, Statistics, and Computer Science,
University of Illinois at Chicago,
322 Science and Engineering Offices (M/C 249),
851 S. Morgan St.,
Chicago, IL 60607-7045}
\email{kwhyte@uic.edu}

\date{\today}

\begin{document}

\begin{abstract}
   We prove that if a proper metric space is quasi-isometric to a finitely generated group and to a space with a horoball over a finitely generated group, then that space is quasi-isometric to a rank-one symmetric space or $\R$.
\end{abstract}

\maketitle

    \section{Introduction}

The isometry groups of negatively curved symmetric spaces correspond, up to index 2, with the rank--$1$ semi-simple Lie groups. A lattice in such a Lie group is a {\it uniform rank--$1$ lattice} if it acts cocompactly on the corresponding symmetric space and a {\it non-uniform rank--$1$ lattice} if it acts with finite co-volume and not cocompactly.  These exist in all dimensions $>1$ by \cite{BorelHS}, and the theory of these lattices has inspired the theory of hyperbolic and relatively hyperbolic groups.

A motivating example of a relatively hyperbolic group pair is a non-uniform rank--$1$ lattice and its parabolic subgroups. The cusped space for such a relatively hyperbolic group pair is quasi-isometric to any group acting geometrically on the associated symmetric space. We prove these spaces are the exceptions.

    \begin{thm} \label{thm:intro_main}
        Let $X'$ be a proper  metric space that is quasi-isometric to a finitely generated group $\Gamma$ and to a space with a horoball.  Then:

        \begin{enumerate}
        \item $X'$ is quasi-isometric to $\R$ or a (non-compact) rank-one symmetric space $X$,
        \item $\Gamma$ is virtually a rank--$1$ uniform lattice in the isometry group of $X$, and
        \item if the space with a horoball is quasi-isometric to the cusped space of a finitely generated relatively hyperbolic group pair $(H, \mathbb{P})$, then $H$ is virtually a non-uniform rank--$1$ lattice in the isometry group of~$X$.
        \end{enumerate}
    \end{thm}

   A {\it space with a horoball} is a space that is obtained by gluing a warped horoball over a finitely generated group onto a proper metric space along the depth-$0$ horosphere; see Section~\ref{sec:rel_hyp}.  A finitely generated group $\Gamma$ is {\it virtually} a lattice in a locally compact group $G$ if there exists a finite-index subgroup $\Gamma' \leq \Gamma$ and a finite normal subgroup $K \triangleleft \Gamma'$ so that $\Gamma'/K$ is isomorphic to a lattice in $G$.

As a degenerate example of the first two parts of Theorem~\ref{thm:intro_main}, if a finitely generated group $\Gamma$ is quasi-isometric to a space with a horoball over a finite group, then $\Gamma$ is 2-ended and acts geometrically on $\mathbb{R}$. The analogue of the third part of Theorem~\ref{thm:intro_main} in the case of horoballs over a finite group is the relatively hyperbolic pair $(A, \{ A,A \})$, where $A$ is a finite group.

    The rank-one symmetric spaces are known to be distinguished among hyperbolic metric spaces in terms of group actions. Caprace--Cornulier--Monod--Tessera~\cite[Theorem E]{CCMT} proved that if a proper geodesic $\delta$-hyperbolic metric space admits both a geometric and a finitely generated geometrically finite group action then, up to a compact normal subgroup, its isometry group is isomorphic to the isometry group of a rank-one symmetric space of noncompact type or its index-two identity component. Trees also admit uniform and non-uniform lattices \cite{treelattices}, but the non-uniform lattices are always infinitely generated \cite[Theorem 3A]{treelattices}, \cite[Proposition 5.16]{BassLubotzky}. Here we restrict to the finitely-generated case. Theorem~\ref{thm:intro_main} extends the result in \cite[Theorem E]{CCMT} about lattices in the same locally compact group to the coarse setting of quasi-isometries.  The recent work on quasi-isometries of solvable groups by Dymarz--Fisher--Xie \cite{DymarzFisherXie-Tukia} is crucial to our results.

    \subsection{Connection to group boundaries}

    Two boundaries of proper visual hyperbolic spaces are quasisymmetric exactly when the spaces are quasi-isometric,  \cite[Section 8]{bonkschramm}.  Thus, the following is an immediate consequence of Theorem~\ref{thm:intro_main}.

     \begin{corollary} If $\Gamma$ is a hyperbolic group whose Gromov boundary is quasisymmetric to the Bowditch boundary of a finitely generated relatively hyperbolic group pair with infinite parabolic subgroups, then this boundary is a sphere, and $\Gamma$ is virtually a uniform rank--$1$ lattice.
     \end{corollary}

     Since $\mathbb{H}^3$ is the only rank--$1$ symmetric space with boundary $S^2$, Theorem~\ref{thm:intro_main} yields the following.

     \begin{corollary}
         If $\Gamma$ is a hyperbolic group with $2$-sphere boundary and $\Gamma$ is quasi-isometric to the cusped space of a relatively hyperbolic group, then $\Gamma$ is virtually Kleinian.
     \end{corollary}

     The boundary perspective also yields the following examples, and there are many other analogs of this.

    \begin{example}[Homeomorphic but not quasisymmetric boundaries]\label{ex:carpet}
       The visual boundary of a hyperbolic 3-manifold with non-empty totally geodesic boundary is homeomorphic to the Sierpinski carpet. The Sierpinski carpet is also the Bowditch boundary of the relatively hyperbolic group pair coming from a hyperbolic 3-manifold with non-empty totally geodesic boundary together with torus cusps. A conjecture of Kapovich and Kleiner \cite{kapovichkleiner} posits these are the only examples of Sierpinski carpet boundaries among hyperbolic and relatively hyperbolic group pairs. Theorem~\ref{thm:intro_main} implies a Sierpinski carpet boundary of any hyperbolic group and the Sierpinski carpet boundary of a relatively hyperbolic group pair are  quasisymmetrically distinct.
    \end{example}

    \subsection{Method of proof.}

    We work in the setting of a proper metric space that is quasi-isometric to both a finitely generated group and to a space with a horoball. Growth implies the combinatorial horoball must be constructed over a finitely generated virtually nilpotent group. We then analyze the structure of a horoball over a finitely generated virtually nilpotent group to prove the following.

    \begin{thm} \label{thm:intro_main_QI} (Theorem~\ref{thm:QI_to_Heintze})
        Let $X$ be a proper metric space. If $X$ is quasi-isometric to a finitely generated group and to a space with a horoball, then $X$ is quasi-isometric to a solvable Lie group $\widehat{N} \rtimes_{\alpha} \R$, where $(\widehat{N},\widehat{d})$ is the asymptotic cone of a finitely generated nilpotent group equipped with a Carnot--Carath\'{e}odory metric, and $\alpha$ acts on $(\widehat{N},\widehat{d})$ by dilation.
    \end{thm}

    Theorem~\ref{thm:intro_main} is then deduced in Section~\ref{sec:QI} as Theorem~\ref{thm:main_thm}.  We apply a theorem of  Dymarz--Fisher--Xie~\cite[Corollary 1.2]{DymarzFisherXie-Tukia} to obtain a geometric group action on $\widehat{N} \rtimes_{\alpha} \R$. Work of Caprace--Cornulier--Monod--Tessera~\cite[Proposition 5.10]{CCMT}
     then implies these spaces must be quasi-isometric to a rank-one symmetric space; this conclusion also follows from Hamenst\"{a}dt~\cite[Theorem 1]{Hamenstadt-Isom}.

    We note that the solvable group $\widehat{N} \rtimes_{\alpha} \R$ is an example of a {\it Heintze group of Carnot type}. These groups were studied by Heintze~\cite{Heintze}, who characterized the connected homogeneous Riemannian manifolds of negative sectional curvature. For additional background on these groups, see \cite{ledonne, LeDonne-book}.

\subsection*{Acknowledgments}

Some motivation for this result arose from a discussion at the American Institute for Mathematics, where a group proved Example \ref{ex:carpet}.  We thank Peter Ha\"issinsky for discussions. ES is thankful for helpful discussions with Jeff Carlson. DG was supported by NSF grant DMS-250646. ES was supported by NSF grant DMS-2204339; GW through NSF DMS-2405033.

    \section{Preliminaries} \label{sec:prelims}

    \subsection{Rank-one symmetric spaces of non-compact type} \label{sec:symmetric}

    A rank-one symmetric space of noncompact type is a quotient $G/K$, where $G$ is a connected, centerless, real semisimple Lie group of rank $1$ and without compact factors, and $K$ is a maximal compact normal subgroup of $G$. The rank-one symmetric spaces of non-compact type are classified as the hyperbolic spaces $\Hy^n_{\K}$, where $\K$ is either $\R$, $\C$, the quaternions, or the octonians, in which case $n=2$. These spaces are the quotients $\SO_0(1,n)/\SO(n)$, $\SU(1,n)/U(n)$, $\Sp(1,n)/\Sp(n)$, and $F_4/\Spin 9$, respectively.
    For background, see \cite[Chapter~19]{Mostow73}, \cite[Chapter~13]{LeDonne-book}, \cite[Chapter II.10]{bridsonhaefliger}.

    A subgroup of the isometry group of a rank-one symmetric space is {\it parabolic} if it fixes a unique point on the visual boundary of the space. A description of the parabolic subgroups is given in \cite[Proposition II.10.28]{bridsonhaefliger}. In particular, a 2-step nilpotent group acts simply transitively on each horosphere of a rank-one symmetric space. These groups are classified as the Heisenberg groups: $\cN_{\K}^n$ acts on each horosphere in $\Hy^n_{\K}$ for $\K = \R, \C$, and $\H$; see \cite[Exercises 13.4.13, 13.4.18]{LeDonne-book} for the structure of $\cN^n_{\K}$. In particular, $\cN_{\C}^2$ is the real Heinsenberg group of dimension 3.  The parabolic subgroups in the Cayley plane case are subgroups of a 15-dimensional nilpotent group.

We note that rank-one symmetric spaces are rigid in the following sense. If a group $\Gamma$ is quasi-isometric to a rank-one symmetric space $X$, then $\Gamma$ virtually acts geometrically on $X$; the $\Hy^2_{\R}$-case is due to \cite{tukia, gabai, cassonjungreis}; the $\Hy^n_{\R}$-case for $n>2$ is due to \cite{Tukia-quasiconformal}; the case of $\Hy^n_{\C}$ is due to \cite{Chow96}; and, the case of the quaternionic and Cayley hyperbolic spaces is due to \cite{Pansu89}. The work of Dymarz--Fisher--Xie~\cite{DymarzFisherXie-Tukia} that we apply here generalizes the results of Tukia~\cite{Tukia-quasiconformal} and Chow~\cite{Chow96} to solvable Lie groups of the form $N \rtimes \R$, where $N$ is a Carnot (nilpotent) group.

    \subsection{Horoballs and relatively hyperbolic groups} \label{sec:rel_hyp}

    There are many notions of horoballs. Classically, they may be defined as a subset of a hyperbolic metric space together with a Busemann function; see \cite{buyaloschroeder}. Alternatively, they can be constructed over a proper metric space to build a $\delta$-hyperbolic metric space with a single boundary point \cite{cannoncooper92}, \cite{grovesmanning}, \cite{bowditch12}. Under natural assumptions, constructions of the latter type yield quasi-isometric spaces \cite{cannoncooper92}, \cite[Appendix A]{grovesmanningsisto}, \cite{healyhruska}. We work in the setting of the following warped products, as this most closely models the geometries we analyze; see Proposition~\ref{prop:GH_convergence}.

      \begin{defn}[Warped product] \label{defn:warped_prod}
        Let $X$ be a complete length space. Let $W(X)$ denote the {\it warped product} $X \times [0,\infty)$ with warped product metric given by the path metric, where the length $\ell(\gamma)$ of a Lipschitz curve $\gamma:[a,b] \rightarrow X \times [0,\infty)$ is given by
            \[ \ell(\gamma):= \int_a^b \sqrt{e^{-2(\gamma_{\R}(t))}|\gamma_X'(t)|^2 + |\gamma_{\R}'(t)|^2} \, dt, \]
        where $\gamma_{X}:[a,b] \rightarrow X$ and $\gamma_{\R}:[a,b] \rightarrow [0,\infty)$ are the projections and the speeds $|\gamma_X'|$ and $|\gamma_{\R}'|$ exist almost everywhere by \cite[Theorem 2.7.6]{buragoburagoivanov}.
    \end{defn}

        If $X$ and $X'$ are quasi-isometric proper length spaces, then $W(X)$ and $W(X')$ are quasi-isometric.

    \begin{defn}[Warped horoball]
        A {\it warped horoball} on a finitely generated group $\Gamma$, denoted $\cH(\Gamma)$, is a warped product $W(X)$, where $X$ is a proper length space equivariantly quasi-isometric to $\Gamma$. The {\it depth-t horosphere} of $\cH(\Gamma)$ is the subspace $X \times \{t\} \subseteq W(X)$.
    \end{defn}

    The Cannon--Cooper horoballs \cite{cannoncooper92} and the combinatorial horoballs of Groves--Manning~\cite{grovesmanning} are quasi-isometric to warped horoballs on a finitely generated group by \cite[Appendix A]{grovesmanningsisto},\cite{healyhruska}. The hyperbolic spike horoballs of Bowditch~\cite{bowditch12} are examples of warped horoballs.

    \begin{remark} \label{rem:QI_to_warped}
        If $P$ is a finitely generated nilpotent group, then $P$ has a finite-index torsion free subgroup $P'$ \cite[Chapter 1]{segal}. The Malcev completion $N_{P'}$ is a connected, simply connected nilpotent Lie group which is quasi-isometric to $P'$ [Rag07, Theorem 2.18], [Mal49]. Thus the warped product $W(N_{P'})$ is a warped horoball $\cH(P)$.
    \end{remark}

   \begin{defn}[Space with a horoball]
      Let $\Gamma$ a finitely generated group. A {\it space with a horoball} is the quotient $\bigl(X \sqcup \cH(\Gamma)\bigr)/\sim$, where $X$ is a proper length metric space, $\cH(\Gamma)$ is a warped horoball on~$\Gamma$, and the depth-0 horosphere $S \subseteq \cH(\Gamma)$ is identified with its image in $X$ under a map $S \rightarrow X$.  We put the natural length metric on the space with a horoball.
    \end{defn}

    We note that a space with a horoball need not contain a quasi-isometrically embedded horoball, as the gluing can distort distances along the depth-0 horosphere. Moreover, the property of containing a quasi-isometrically embedded horoball is not strong enough for our purposes: every infinite $\delta$-hyperbolic group contains a quasi-isometrically embedded horoball over a finite group. The next remark records precisely the property that we will utilize.

  \begin{remark}  \label{rem:balls_the_same}
    Let $Y$ be a space with a horoball $\cH(\Gamma)$, and let $i:\cH(\Gamma) \rightarrow Y$.  For all $r>0$ there exists $t \in \N$ so that each $r$-ball $B_r(x)$ in $\cH(\Gamma)$ at depth greater than $t$ is isometric to the $r$-ball $B_r(i(x))$ in~$Y$.
  \end{remark}

    Spaces with a horoball arise naturally in the following context. A {\it group pair} $(H,\bP)$ is a group $H$ together with a collection of subgroups $\bP$ of $H$. We assume that $\bP$ is a finite collection and that $H$ and each group $P \in \bP$ are finitely generated.

\begin{defn}[Cusped space]
        Let $(H, \bP)$ be a group pair, and let $X_H$ be a proper length metric space that $H$ acts on geometrically. For each $P \in \bP$, let $W(P)$ be a warped horoball on $P$.  Let $Hx$ denote the $H$-orbit of a point $x \in X_H$. For each $P \in \mathbb{P}$, choose a quasi-isometry $q_P:S_0 \rightarrow P$, where $S_0$ is the depth-0 horosphere of $W(P)$ and is equipped with the intrinsic metric (not the metric induced by $W(P)$). Then the {\it cusped space} $X(H,\bP)$ is obtained  by identifying $S_0$ to the orbit $hPx$ via the map $h \circ q_P$, for each distinct coset $hP$ of each $P\in \mathbb{P}$.
        A group pair $(H,\bP)$ is \emph{relatively hyperbolic} if the cusped space $X(H, \bP)$ is $\delta$-hyperbolic.
    \end{defn}

   Note that this definition of relatively hyperbolic group pair is equivalent to other definitions in the literature.  In particular, a cusped space above is quasi-isometric to the space in \cite{grovesmanning}.  See also  \cite{hruska10} and \cite{healyhruska}.

    \begin{defn}[Bowditch boundary]
         Let $(H,\bP)$ be a relatively hyperbolic group pair, and suppose each $P \in \bP$ is an infinite group. The \emph{Bowditch boundary} $\partial (H, \bP)$ is the visual boundary of a cusped space $X(H, \bP)$.
    \end{defn}

    Bowditch~\cite{bowditch12} proved that the homeomorphism type of the Bowditch boundary is independent of the choices above.


    \subsection{Finitely generated nilpotent groups}

    Crucial to the results in this paper is the description of the asymptotic cone of a finitely generated nilpotent group given by Pansu~\cite{pansu_cones}, which extends the work of Gromov~\cite{gromov-poly}.

    \begin{defn}[Gromov--Hausdorff convergence]
        Let $X$ and $Y$ be compact metric spaces. The {\it Gromov--Hausdorff distance} between $X$ and $Y$, denoted $d_{GH}(X,Y)$, is the infimum of the Hausdorff distance $d_{Haus}\bigl(f(X), g(Y)\bigr)$, taken over all compact metric spaces $M$ and isometric embeddings $f:X \rightarrow M$ and $g:Y \rightarrow M$.

        A {\it pointed metric space} is a pair $(X,p)$ of a metric space $X$ together with $p \in X$. A sequence of pointed proper metric spaces $(X_n,p_n)$ converges to a pointed proper metric space $(X,p)$ in the {\it pointed Gromov--Hausdorff sense} if for each $R>0$ the sequence of closed $R$-balls around $p_n$ in $X_n$ converges to the closed $R$-ball around $p$ in $X$ with respect to the Gromov--Hausdorff distance.
    \end{defn}

Pansu \cite{pansu_cones} proved that the scaled metrics on a nilpotent simply-connected Lie group $G$ converge in the pointed Gromov-Hausdorff sense to a graded nilpotent Lie group $(G_\infty,d_\infty)$.  The Lie algebra $\mathfrak{g}_\infty$ of $G_\infty$ is the associated Carnot graded Lie algebra to $\mathfrak {g}$.

    \begin{thm} \cite[Theorem 10; Section B]{pansu_cones} \label{thm:pansu_results}
        Let $N$ be a connected, simply connected nilpotent Lie group.
        The asymptotic cone $\widehat{N}$ of $N$ is unique and has the structure of a stratified (nilpotent) Lie group, whose limiting metric $\widehat{d}$ is a left-invariant Carnot--Carath\'{e}odory metric. Moreover, $(\widehat{N},\widehat{d})$ admits a $1$--parameter family $\alpha_t$ of automorphisms that scale the metric $\widehat{d}$ by $e^{-t}$.
    \end{thm}

    \begin{notation}[Associated solvable group] \label{nota:Heintze_group}
        Let $N$ be a connected, simply connected nilpotent Lie group. As in Theorem~\ref{thm:pansu_results}, the asymptotic cone $(\widehat{N}, \widehat{d})$ admits a 1-parameter family of automorphisms $\{\alpha_t\,|\, t \in \R\}$. The associated warped product structure on $\widehat{N} \rtimes_{\alpha} \R$ with warping function $f(t)=e^{-t}$ therefore gives rise to a solvable Lie group $\widehat{N} \rtimes_{\alpha} \R$.
        We call the group $\widehat{N} \rtimes_{\alpha} \R$ together with the warped product metric the {\it associated solvable group} to $N$.
        We use $Id_{S(N)}$ to denote the identity in $S(N):=\widehat{N} \rtimes_{\alpha} \R$.
    \end{notation}

    \section{Horoball structure}

    Growth of balls in a horoball, as analyzed by Dahmani--Yaman~\cite{dahmaniyaman}, restricts the spaces we must consider to horoballs over nilpotent groups.

    \begin{lemma}[Reduction to nilpotent groups] \label{lem:must_be_nilpotent}
        If a space with a horoball $\cH(\Lambda)$ over an infinite group $\Lambda$ is quasi-isometric to a finitely generated group, then $\Lambda$ is a finitely generated virtually nilpotent group.
    \end{lemma}
    \begin{proof}
        Suppose there is a quasi-isometry from a space with a horoball $\cH(\Lambda)$ to a finitely generated group~$\Gamma$. The group $\Gamma$ has {\it bounded geometry}, meaning there exists a function $f:\R_+\rightarrow \R_+$ so that for all $R>0$ every ball of radius $R$ can be covered by $f(R)$ balls of radius 1. Thus, the horoball $\cH(\Lambda)$ has bounded geometry, as does the corresponding quasi-isometric horoball defined by Bowditch~\cite{bowditch12} and considered by Dahmani--Yaman. Therefore, $\Lambda$ is a finitely generated virtually nilpotent group by \cite[Theorem 0.1]{dahmaniyaman} and its proof.
    \end{proof}

    We note that in the case that a proper metric space $X$ is quasi-isometric to a finitely generated group and to the cusped space of a relatively hyperbolic group pair $(H, \bP)$, then Lemma~\ref{lem:must_be_nilpotent} also follows from the quasisymmetric doubling condition with respect to the boundaries $\p X$ and $\p(H, \bP)$ by \cite[Theorem 9.2]{bonkschramm}, \cite[Proposition 4.5]{mackaysisto}, \cite[Theorem 10.18]{heinonen-Lectures}.

    \begin{prop}[Convergence to solvable group] \label{prop:GH_convergence}
        Let $N$ be a connected, simply connected nilpotent Lie group, equipped with a left-invariant metric, and let $W(N)$ be the corresponding warped product.
        Let $w_i = (n_i,i)$ for $i \in \N$ be a sequence of points in $ W(N)$. The sequence of pointed spaces $(W(N),w_i)$ converges in the pointed Gromov--Hausdorff sense to the associated solvable group $(\widehat{N} \rtimes_{\alpha} \R, Id_{S(N)})$.
    \end{prop}
    \begin{proof}
        Let $N$ be a connected, simply connected nilpotent Lie group, and let $d$ be a left-invariant metric on $N$. By Theorem~\ref{thm:pansu_results}, the metric spaces $(N, e^{-i}d)$ converge to the asymptotic cone $(\widehat{N},\widehat{d})$ of $N$ in the sense of pointed Gromov--Hausdorff convergence.

        Fix a basepoint $x \in N$, and let $w_i = (x,i) \in W(N)$.  Since $N$ is a group and the metric $d$ is left-invariant, it suffices to show that for any $r>0$ the ball $B_r(w_i) \subseteq W(N)$ converges to $B_r(Id_{S(N)}) \subseteq \widehat{N} \rtimes_{\alpha} \R$ in the Gromov-Hausdorff sense.  Note that because $\widehat{N} \rtimes_{\alpha} \R$ is a group with a left-invariant metric, there is only one such ball up to isometry.

        Let $R = e^{2r} \cdot r$, and let $\varepsilon>0$.  By Theorem~\ref{thm:pansu_results}, there exists $n\in \N$ so that for $i \geq n$ $$d_{GH}\bigl(B_R(N,e^{-i}d),B_R(\widehat{N},\widehat{d})\bigr) < \varepsilon.$$
        Thus, there is a metric space $Z_i = B_R(N,e^{-i}d) \cup B_R(\widehat{N},\widehat{d})$ in which these balls are at Hausdorff distance less than $ \varepsilon$.

    Consider the warped product $W_i = W(Z_i) \cong Z_i \rtimes [0,\infty)$ as in Definition~\ref{defn:warped_prod}.   The warped products $W(B_R(N,e^{-i}d))$ and  $W(B_R(\widehat{N},\widehat{d}))$ are isometrically embedded in $W_i$, and they are at Hausdorff distance less than~$\varepsilon$ in $W_i$. Thus, the warped products $W(B_R(N,e^{-i}d))$ converge to the warped product $W(B_R(\widehat{N},\widehat{d}))$.

        Finally, let $t > r$ and let $(x,t) \in W(N)$.  Let $B_R(N,e^{t-r}d)$ denote the ball of radius $R$ about $x$ in $(N,e^{t-r}d)$.  The choice of $R$ is so that the ball of radius $r$ about $(x,t)$ in $W(N)$ is contained in $B_R(N,e^{t-r}d) \times [t-r,\infty)$.
        To see this, let $B_r(x,t)$ be this ball in $W(N)$.  The map $h : W(N) \to [0,\infty)$ given by projection to the second factor is distance non-increasing, so $B_r(x,t) \subseteq h^{-1}\left( [t-r,\infty) \right)$.  Suppose that $(x',t') \in B_r(x,t)$.  Take a path between $(x,t)$ and $(x',t')$ and project it to $N \times \{t_{\max}\}$, where $t _{\max} = \max \{t,t'\}$.  The projection does not increase the length of this path, so the distance between $(x,t_{\max})$ and $(x',t_{\max})$ in the metric $e^{-t_{\max}}d$ on $N \times \{ t_{\max} \}$ is at most $r$.  Since $t_{\max} \le t + r$, it follows that the distance between $(x,t-r)$ and $(x',t-r)$ in the metric $e^{t-r}d$ on $N \times \{ t-r \}$ is at most $e^{2r} \cdot r = R$.  This proves the assertion about the ball of radius $r$ being contained in $B_R(N,e^{t-r}d) \times [t-r,\infty)$ and completes the proof.
    \end{proof}

    We will use the following corollary to Proposition \ref{prop:GH_convergence}.

    \begin{corollary} \label{cor:QI_to_S}
        Let $N$ be a connected, simply connected nilpotent Lie group, and let $S(N) = \widehat{N} \rtimes_{\alpha} \R$ be the associated solvable group.
        Let $p \in N$, and let $w_i = (p,i) \in W(N)$ for $i \in \N$. There exist constants $K \geq 1$ and $C \geq 0$ so that for all $R>0$ there exists $k = k(R) \in \N$ and a $(K,C)$-quasi-isometry $h_R: B_R(w_k) \rightarrow B_R(Id_{S(N)})$ with $h_R(w_k) = Id_{S(N)}$. \qed
    \end{corollary}

    \section{Quasi-isometry of spaces} \label{sec:QI}

    \begin{thm} \label{thm:QI_to_Heintze}
         Let $X$ be a proper metric space. If $X$ is quasi-isometric to a finitely generated group and to a space with a horoball over an infinite group, then $X$ is quasi-isometric to a solvable Lie group $\widehat{N} \rtimes_{\alpha} \R$, where $(\widehat{N},\widehat{d})$ is the asymptotic cone of a finitely generated nilpotent group equipped with a Carnot--Carath\'{e}odory metric, and $\alpha$ acts on $(\widehat{N},\widehat{d})$ by dilation.
    \end{thm}
    \begin{proof}
        Let $X$ be a proper metric space. We may suppose that $X$ is the vertex set of a Cayley graph of a finitely generated group $\Gamma$ constructed with respect to a finite generating set and $X$ is equipped with the induced metric. Let $Y$ be a space with a horoball $\cH(\Lambda)$, and let $i:\cH(\Lambda) \rightarrow Y$ be the inclusion. Let $\phi:X \rightarrow Y$ be a quasi-isometry.  There exists a quasi-isometry $\psi:\cH(\Lambda) \rightarrow W(N)$, where $N$ is a connected, simply connected nilpotent Lie group and $W(N)$ is the warped product by Lemma~\ref{lem:must_be_nilpotent} and Remark~\ref{rem:QI_to_warped}. Let $\widehat{N} \rtimes_{\alpha} \R$ denote the associated solvable group as in Notation~\ref{nota:Heintze_group}.

        Fix $x \in X$, and let $B_r(x)$ denote the closed ball of radius $r$ about $x$. We define a sequence of $(K,C)$-quasi-isometric embeddings $f_r:B_r(x) \rightarrow \widehat{N} \rtimes_{\alpha} \R$ for fixed $K \geq 1$ and $C \geq 0$ that has a subsequence that converges to a quasi-isometry from $X$ to $\widehat{N} \rtimes_{\alpha} \R$. To do this, we use the action of the group $\Gamma$ to move the ball $B_r(x)$ in $X$ so that its image under the composition of quasi-isometries $\psi\circ \phi$ is high enough in the warped product $W(N)$ to yield a quasi-isometric embedding into the solvable group $\widehat{N} \rtimes_{\alpha} \R$ via Corollary~\ref{cor:QI_to_S}. The details are as follows.

        Fix $p \in N$, and let $w_i = (p,i) \in W(N)$ for $i \in \N$. By Corollary~\ref{cor:QI_to_S} there exist constants $K_0 \geq 1$ and $C_0 \geq 0$ so that for all $R > 0$ there exists $k=k(R) \in \N$ and a $(K_0, C_0)$-quasi-isometry $h_R:B_R(w_k) \rightarrow B_R(Id_{S(N)}) \subseteq \widehat{N} \rtimes_{\alpha} \R$ with $h_R(w_k) = Id_{S(N)}$.

    Let $K_1 \geq 1$ and $C_1 \geq 0$ be the maximum of the quasi-isometry constants of the maps $\phi$ and $\psi$ defined above. We may also suppose that $C_1$ is large enough so that $\bigcup_{g \in \Gamma} B_{C_1}(gx) = X$.  There exists $D = D(K_1, C_1) \in \R$ so that for all $w \in W(N)$ there exists $g \in \Gamma$ so that $d\bigl( w, \psi(\phi(gx))\bigr) \leq D$.  Suppose $x' \in X$ and $B=B_r(x')$ is a ball of radius $r$ so that the image $\phi(B)$ is contained in $i(\cH(\Lambda)) \subseteq Y$ and is high enough in the glued horoball $i(\cH(\Lambda))$ to be contained in a subset isometric to a subset of $\cH(\Lambda)$; such a height exists by Remark~\ref{rem:balls_the_same}. Then, there exists $D' = D'(K_1, C_1, r) \in \R$ so that if $\psi(\phi(B)) \subseteq W(N)$ and $x_1,x_2 \in B$, then $\psi(\phi(B)) \subseteq B_{D'}(\psi(\phi(x')))$ and the geodesic in $W(N)$ between the images $\psi(\phi(x_1))$ and $\psi(\phi(x_2))$ is contained in this ball $B_{D'}(\psi(\phi(x')))$. Let $R = D+D'$ and $k = k(R) \in \N$. We may assume $k$ is sufficiently large so that there exists $g_r \in \Gamma$ so that $(\psi \circ \phi \circ g_r)(B_r(x)) \subseteq B_R(w_k) \subseteq W(N)$ and $\phi\bigl(g_r(B_r(x))\bigr)$ is contained in $i(\cH(\Lambda)) \subseteq Y$ and is high enough in the glued horoball $i(\cH(\Lambda))$ to be contained in a subset isometric to a subset of $\cH(\Lambda)$.

        Replace $\psi\circ \phi \circ g_r$ if necessary to assume $x$ is mapped to $w_k$. That is, let $\rho_r:B_r(x) \rightarrow B_R(w_k)$ be given by $\rho_r(x') = (\psi\circ \phi \circ g_r)(x')$ for all $x' \neq x$ and $\rho_r(x) = w_k$. Then, $\rho_r$ is still a uniform quasi-isometric embedding of $B_r(x)$ to $B_R(w_k)$ and has the additional property that $\rho_r(x) = w_k$.
        Therefore, there exist $K \geq 1$ and $C \geq 0$ so that the map
        \begin{equation} \label{eqn:QI}
            f_r:= h_R \circ \rho_r:B_r(x) \rightarrow B_R(Id_{S(N)}) \subseteq \widehat{N} \rtimes_{\alpha} \R
        \end{equation}
        is a $(K,C)$-quasi-isometric embedding with $f_r(x) = Id_{S(N)}$.

        Adjust these maps $\{f_r\,|\, r \in \N\}$ so that images lie in a net in $\widehat{N} \rtimes_{\alpha} \R$ and obtain a subsequence that converges. Indeed, let $\cS$ be a net in $\widehat{N} \rtimes_{\alpha} \R$, meaning that there exists $\delta>0$ so that $\widehat{N} \rtimes_{\alpha} \R = \bigcup_{s \in \cS} B_\delta(s)$ and $d(s,s') \geq \delta$ for all $s,s' \in \cS$ with $s\neq s'$. Suppose $Id_{S(N)} \in \cS$. We may assume that the image of $f_r$ is contained in $\cS$ for each $r$. For each $r$, there are finitely many $(K,C)$-quasi-isometric embeddings $f$ from $B_r(x)$ to $\cS \subseteq \widehat{N} \rtimes_{\alpha} \R$ with $f(x) = Id_{S(N)}$ since the intersection of $\cS$ and $B_{Kr+C}(Id_{S(N)})$ has finitely many points. Thus, a diagonal argument implies that a subsequence of the maps $f_r$ converges to a map $f:X \rightarrow \widehat{N} \rtimes_{\alpha} \R$ with $f(x) = Id_{S(N)}$.

        The map $f$ is a quasi-isometric embedding since each map $f_r$ is a $(K,C)$-quasi-isometric embedding.
        The map $f$ is almost onto since the maps in the definition of $f_r$ are quasi-isometries, and if $\rho:Z \rightarrow Z'$ is a $(K,C)$-quasi-isometry, then $\rho|_{B_{Kr+KC(z)}}$ is almost onto $B_{r-C}(\rho(z))$. Therefore, $f$ is a quasi-isometry, and $X$ and $\widehat{N} \rtimes_{\alpha} \R$ are quasi-isometric.
  \end{proof}

    We now prove our main Theorem~\ref{thm:intro_main} using a theorem of Dymarz--Fisher--Xie~\cite{DymarzFisherXie-Tukia} and a result of Caprace--Cornulier--Monod--Tessera~\cite{CCMT} that generalizes   Hamenst\"adt~\cite[Theorem 1]{Hamenstadt-Isom}. We recall the statement here for convenience.

    \begin{thm} \label{thm:main_thm}
        Let $X'$ be a proper metric space that is quasi-isometric to a finitely generated group $\Gamma$ and to a space with a horoball.  Then:
        \begin{enumerate}
        \item \label{item:1} $X'$ is quasi-isometric to a rank-one symmetric space $X$ (or $\R$),
        \item \label{item:2} $\Gamma$ is virtually a rank--$1$ uniform lattice in the isometry group of $X$, and
        \item \label{item:3} if the space with a horoball is quasi-isometric to the cusped space of a relatively hyperbolic group pair $(H, \mathbb{P})$, then $H$ is virtually a rank--$1$ non-uniform lattice in the isometry group of~$X$.
        \end{enumerate}
    \end{thm}
    \begin{proof}
         Suppose $X'$ is a proper metric space quasi-isometric to a finitely generated group $\Gamma$ and to a space with a horoball. If this horoball is over a finite group, then $\Gamma$ is two-ended, so $X'$ is quasi-isometric to $\R$.
         So we now assume that the horoball is over an infinite group.  Then, the group $\Gamma$ is quasi-isometric to a solvable group $\widehat{N} \rtimes_{\alpha} \R$ as in Theorem~\ref{thm:QI_to_Heintze}. Therefore, by \cite[Corollary 1.2]{DymarzFisherXie-Tukia} there exists a left-invariant Riemannian metric~$g_0$ on $\widehat{N} \rtimes_{\alpha} \R$ so that $\Gamma$ acts geometrically on $(\widehat{N} \rtimes_{\alpha} \R, g_0)$. The space $(\widehat{N} \rtimes_{\alpha} \R, g_0)$ is a proper geodesic hyperbolic metric space. Thus, by \cite[Proposition 5.10]{CCMT}, if $G$ is the isometry group of $(\widehat{N} \rtimes_{\alpha} \R, g_0)$, then $G$ has a unique maximal compact normal subgroup $W$ so that $G/W$ is isomorphic to the full group of isometries of a rank-one symmetric space or its identity component. Since $\Gamma/(\Gamma \cap W)$ is virtually a uniform lattice in $G/W$, the group $\Gamma$ virtually acts geometrically on a rank-one symmetric space, proving Items~(\ref{item:1} and \ref{item:2}). Item~\ref{item:3} then follows from \cite[Theorem 1.3]{healyhruska}.
    \end{proof}

\bibliographystyle{alpha}
\bibliography{refs.bib}

\begin{thebibliography}{CCMT15}

\bibitem[BBI01]{buragoburagoivanov}
Dmitri Burago, Yuri Burago, and Sergei Ivanov.
\newblock {\em A course in metric geometry}, volume~33 of {\em Graduate Studies
  in Mathematics}.
\newblock American Mathematical Society, Providence, RI, 2001.

\bibitem[BH99]{bridsonhaefliger}
Martin~R. Bridson and Andr{\'e} Haefliger.
\newblock {\em Metric spaces of non-positive curvature}, volume 319 of {\em
  Grundlehren der Mathematischen Wissenschaften [Fundamental Principles of
  Mathematical Sciences]}.
\newblock Springer-Verlag, Berlin, 1999.

\bibitem[BHC62]{BorelHS}
Armand Borel and Harish-Chandra.
\newblock Arithmetic subgroups of algebraic groups.
\newblock {\em Ann. of Math.}, 75(3):485--535, 1962.

\bibitem[BL01]{BassLubotzky}
Hyman Bass and Alexander Lubotzky.
\newblock {\em Tree lattices}, volume 176 of {\em Progress in Mathematics}.
\newblock Birkh\"auser Boston, Inc., Boston, MA, 2001.
\newblock With appendices by Bass, L. Carbone, Lubotzky, G. Rosenberg and J.
  Tits.

\bibitem[Bow12]{bowditch12}
B.~H. Bowditch.
\newblock Relatively hyperbolic groups.
\newblock {\em Internat. J. Algebra Comput.}, 22(3):1250016, 66, 2012.

\bibitem[BS00]{bonkschramm}
M.~Bonk and O.~Schramm.
\newblock Embeddings of {Gromov} hyperbolic spaces.
\newblock {\em Geom. Funct. Anal.}, 10(2):266--306, 2000.

\bibitem[BS07]{buyaloschroeder}
Sergei Buyalo and Viktor Schroeder.
\newblock {\em Elements of asymptotic geometry}.
\newblock EMS Monographs in Mathematics. European Mathematical Society (EMS),
  Z\"{u}rich, 2007.

\bibitem[CC92]{cannoncooper92}
J.~W. Cannon and Daryl Cooper.
\newblock A characterization of cocompact hyperbolic and finite-volume
  hyperbolic groups in dimension three.
\newblock {\em Trans. Amer. Math. Soc.}, 330(1):419--431, 1992.

\bibitem[CCMT15]{CCMT}
Pierre-Emmanuel Caprace, Yves Cornulier, Nicolas Monod, and Romain Tessera.
\newblock Amenable hyperbolic groups.
\newblock {\em J. Eur. Math. Soc. (JEMS)}, 17(11):2903--2947, 2015.

\bibitem[Cho96]{Chow96}
Richard Chow.
\newblock Groups quasi-isometric to complex hyperbolic space.
\newblock {\em Trans. Amer. Math. Soc.}, 348(5):1757--1769, 1996.

\bibitem[CJ94]{cassonjungreis}
Andrew Casson and Douglas Jungreis.
\newblock Convergence groups and {S}eifert fibered {$3$}-manifolds.
\newblock {\em Invent. Math.}, 118(3):441--456, 1994.

\bibitem[DFX23]{DymarzFisherXie-Tukia}
Tullia Dymarz, David Fisher, and Xiangdong Xie.
\newblock A fibered {T}ukia theorem for nilpotent {L}ie groups.
\newblock {\em Ann. Fenn. Math.}, 48(2):653--680, 2023.

\bibitem[DY05]{dahmaniyaman}
Fran\c{c}ois Dahmani and Asl\i Yaman.
\newblock Bounded geometry in relatively hyperbolic groups.
\newblock {\em New York J. Math.}, 11:89--95, 2005.

\bibitem[Gab92]{gabai}
David Gabai.
\newblock Convergence groups are {F}uchsian groups.
\newblock {\em Ann. of Math. (2)}, 136(3):447--510, 1992.

\bibitem[GM08]{grovesmanning}
Daniel Groves and Jason~Fox Manning.
\newblock Dehn filling in relatively hyperbolic groups.
\newblock {\em Israel J. Math.}, 168:317--429, 2008.

\bibitem[GMS19]{grovesmanningsisto}
Daniel Groves, Jason~Fox Manning, and Alessandro Sisto.
\newblock Boundaries of {D}ehn fillings.
\newblock {\em Geom. Topol.}, 23(6):2929--3002, 2019.

\bibitem[Gro81]{gromov-poly}
Mikhael Gromov.
\newblock Groups of polynomial growth and expanding maps.
\newblock {\em Inst. Hautes \'Etudes Sci. Publ. Math.}, (53):53--73, 1981.

\bibitem[Ham09]{Hamenstadt-Isom}
Ursula Hamenst\"adt.
\newblock Isometry groups of proper hyperbolic spaces.
\newblock {\em Geom. Funct. Anal.}, 19(1):170--205, 2009.

\bibitem[Hei74]{Heintze}
Ernst Heintze.
\newblock On homogeneous manifolds of negative curvature.
\newblock {\em Math. Ann.}, 211:23--34, 1974.

\bibitem[Hei01]{heinonen-Lectures}
Juha Heinonen.
\newblock {\em Lectures on analysis on metric spaces}.
\newblock Universitext. New York, NY: Springer, 2001.

\bibitem[HH]{healyhruska}
Burns Healy and G.~Christopher Hruska.
\newblock Cusped spaces and quasi-isometries of relatively hyperbolic groups.
\newblock arXiv:2010.09876.

\bibitem[Hru10]{hruska10}
G.~Christopher Hruska.
\newblock Relative hyperbolicity and relative quasiconvexity for countable
  groups.
\newblock {\em Algebr. Geom. Topol.}, 10(3):1807--1856, 2010.

\bibitem[KK00]{kapovichkleiner}
Michael Kapovich and Bruce Kleiner.
\newblock Hyperbolic groups with low-dimensional boundary.
\newblock {\em Ann. Sci. \'Ecole Norm. Sup. (4)}, 33(5):647--669, 2000.

\bibitem[LD17]{ledonne}
Enrico Le~Donne.
\newblock A primer on {C}arnot groups: homogenous groups,
  {C}arnot-{C}arath\'eodory spaces, and regularity of their isometries.
\newblock {\em Anal. Geom. Metr. Spaces}, 5(1):116--137, 2017.

\bibitem[LD25]{LeDonne-book}
Enrico Le~Donne.
\newblock {\em Metric {L}ie groups---{C}arnot-{C}arath\'eodory spaces from the
  homogeneous viewpoint}, volume 306 of {\em Graduate Texts in Mathematics}.
\newblock Springer, Cham, [2025] \copyright 2025.

\bibitem[Lub95]{treelattices}
Alexander Lubotzky.
\newblock Tree-lattices and lattices in {L}ie groups.
\newblock In {\em Combinatorial and geometric group theory ({E}dinburgh,
  1993)}, volume 204 of {\em London Math. Soc. Lecture Note Ser.}, pages
  217--232. Cambridge Univ. Press, Cambridge, 1995.

\bibitem[Mos73]{Mostow73}
G.~D. Mostow.
\newblock {\em Strong rigidity of locally symmetric spaces}.
\newblock Princeton University Press, Princeton, N.J.; University of Tokyo
  Press, Tokyo, 1973.
\newblock Annals of Mathematics Studies, No. 78.

\bibitem[MS20]{mackaysisto}
John~M. Mackay and Alessandro Sisto.
\newblock Quasi-hyperbolic planes in relatively hyperbolic groups.
\newblock {\em Ann. Acad. Sci. Fenn., Math.}, 45(1):139--174, 2020.

\bibitem[Pan83]{pansu_cones}
Pierre Pansu.
\newblock Croissance des boules et des g\'eod\'esiques ferm\'ees dans les
  nilvari\'et\'es.
\newblock {\em Ergodic Theory Dynam. Systems}, 3(3):415--445, 1983.

\bibitem[Pan89]{Pansu89}
Pierre Pansu.
\newblock M\'{e}triques de {C}arnot-{C}arath\'{e}odory et quasiisom\'{e}tries
  des espaces sym\'{e}triques de rang un.
\newblock {\em Ann. of Math. (2)}, 129(1):1--60, 1989.

\bibitem[Seg83]{segal}
Daniel Segal.
\newblock {\em Polycyclic groups}, volume~82 of {\em Cambridge Tracts in
  Mathematics}.
\newblock Cambridge University Press, Cambridge, 1983.

\bibitem[Tuk86]{Tukia-quasiconformal}
Pekka Tukia.
\newblock On quasiconformal groups.
\newblock {\em J. Analyse Math.}, 46:318--346, 1986.

\bibitem[Tuk88]{tukia}
Pekka Tukia.
\newblock Homeomorphic conjugates of {F}uchsian groups.
\newblock {\em J. Reine Angew. Math.}, 391:1--54, 1988.

\end{thebibliography}

\end{document}